\documentclass[11pt,A4]{amsart}
\usepackage{ifthen,latexsym,amssymb,amsmath}
\newtheorem{theorem}{Theorem}
\newtheorem{lemma}[theorem]{Lemma}
\newtheorem{corrolary}[theorem]{Corrolary}

\begin{document}

\newcommand{\sgn}{\text{sgn}}

\author{Bart\l{}omiej Bzd\c{e}ga}

\address{Str\'o\.zy\'nskiego 15A/20 \\ 60-688 Pozna\'n, Poland}

\email{exul@wp.pl}



\title{Inclusion-exclusion polynomials with large coefficients}

\maketitle

\begin{abstract}
We prove that for every positive integer $k$ there exist an inclusion-exclusion polynomial $Q_{\{q_1,q_2,\ldots,q_k\}}$ with the height at least $c^{2^k}\prod_{j=1}^{k-2}q_j^{2^{k-j-1}-1}$, where $c$ is a positive constant and $q_1<q_2<\ldots<q_k$ are pairwise coprime and arbitrary large.
\end{abstract}

\section{Introduction}

The $n$th cyclotomic polynomial is the unique monic polynomial over integers, which roots are all the $n$th primitive roots of unity. It is well-known that if $n=p_1p_2\ldots p_k$, where $p_1,p_2,\ldots,p_k$ are distinct primes, then
\begin{equation} \label{f}
\Phi_n(x) = \frac{(1-x^n)\cdot\prod_{1\le j_1<j_2\le k}\left(1-x^{n/(p_{j_1}p_{j_2})}\right)\cdot\ldots}{\prod_{1\le j_1\le k}(1-x^{n/p_{j_1}})\cdot\prod_{1\le j_1<j_2<j_3\le k}\left(1-x^{n/(p_{j_1}p_{j_2}p_{j_3})}\right)\cdot\ldots}
\end{equation}

Bachman \cite{Bachman-TernaryInEx} defined a slightly more general class of polynomials, called the inclusion-exclusion polynomials. If we replace primes $p_1,p_2,\ldots,p_k$ by pairwise coprime numbers $q_1,q_2,\ldots,q_k$ in the formula above and put $m=q_1q_2\ldots q_k$ instead of $n$, then we receive the definition of the inclusion-exclusion polynomial $Q_\rho$, where $\rho=\{q_1,q_2,\ldots,q_k\}$.

We can expect that the properties of inclusion-exclusion polynomials and cyclotomic polynomials are similar. In particular, we may use the same methods to bound the coefficients of polynomials of these both classes, as long as we do not need the assumption that numbers $p_1,p_2,\ldots,p_k$ are prime and use only formula (\ref{f}). In this note we present an example of this situation.

Troughout the paper we set $n=p_1p_2\ldots p_k$, $m=q_1q_2\ldots q_k$ and $\rho=\{q_1,q_2\ldots,q_k\}$. We also assume that $p_1<p_2<\ldots<p_k$ and $q_1<q_2<\ldots<q_k$.

Let $A_n$ be the largest magnitude of coefficients of given cyclotomic polynomial $\Phi_n$. Similarly we define $A_\rho$ for the polynomial $Q_\rho$. Put
$$M_n=\prod_{j=1}^{k-2}p_j^{2^{k-j-1}-1} \quad \text{and} \quad M_\rho=\prod_{j=1}^{k-2}q_j^{2^{k-j-1}-1}.$$

It was proved in \cite{Bzdega-Height} that $(A_n/M_n)^{2^{-k}} < C+o_k(1)$, where $o_k(1)\to0$ with $k\to\infty$. In the proof it is never used that $p_1,p_2,\ldots,p_k$ are primes, so this estimation is true also for inclusion-exclusion polynomials. More precisely, the following holds.

\begin{theorem} \label{thm-upp}
We have
$$(A_\rho/M_\rho)^{2^{-k}} < C+o_k(1),$$
where $o_k(1)\to0$ with $k\to\infty$.
\end{theorem}

The aim of this paper is to give an example of inclusion-exclusion polynomials $Q_\rho$ for which $q_1,q_2,\ldots,q_k$ are arbitrary large and the opposite inequality holds, with a smaller constant $c$ replacing $C$. Our main result is the following theorem.

\begin{theorem} \label{thm-low}
There exist a positive constant $c$ such that for every $k$ and every $N$ there exist inclusion-exclusion polynomial $Q_\rho$ for which $q_1>N$ and
$$(A_\rho/M_\rho)^{2^{-k}} > c+o_k(1),$$
where $o_k(1)\to0$ with $k\to\infty$.
\end{theorem}

\section{Proof of Theorem \ref{thm-upp}}

Bateman, Pomerance and Vaughan (\cite{BatemanPomeranceVaughan-Size}, Lemma 5, p. 188) proved that if $r$ is a positive integer and $p_j\equiv2r\pm1\pmod{4r}$ for $j=1,2,\ldots,k$, then $A_n\ge r^{2^{k-1}}/n.$

By checking that their proof uses only formula (\ref{f}) and that the assumption that $p_1,p_2,\ldots,p_k$ are primes is not required, we deduce the following inclusion-exclusion version of this lemma.

\begin{lemma} \label{lem-BPV}
Let $r$ be a positive integer. If $q_j\equiv2r\pm1\pmod{4r}$ for $j=1,2,\ldots,k$, then $A_\rho \ge r^{2^{k-1}}/m$.
\end{lemma}

Now we are ready to prove the main result of this paper.

\begin{proof}[Proof of Theorem \ref{thm-low}]
Let $r=Nk!$ and $q_j=(4j-2)r+1$ for $j=1,2,\ldots,k$. First we check that the numbers $q_1,q_2,\ldots,q_k$ are pairwise coprime. By the Euclid algorithm we have
$$(q_i,q_j) = ((4i-2)r+1, (4j-2)r+1) = (4(i-j)r, (4j-2)r+1) = 1,$$
because every prime divisor of $4(i-j)r$ divides $N$ or is not greater than $k$ and the number $(4j-2)r+1$ has no such prime divisors.

By Lemma \ref{lem-BPV} we have
\begin{align*}
(A_\rho/M_\rho)^{2^{-k}} & > \left(\frac{r^{2^{k-1}}/m}{\prod_{j=1}^{k-2}q_j^{2^{k-j-1}-1}}\right)^{2^{-k}} = \left(\frac{r^{2^{k-1}}/q_k}{\prod_{j=1}^{k-1}q_j^{2^{k-j-1}}}\right)^{2^{-k}} \\
& = \left(\frac{r}{q_k}\prod_{j=1}^{k-1}\left(\frac{r}{q_j}\right)^{2^{k-j-1}}\right)^{2^{-k}} = \prod_{j=1}^\infty(4j-2)^{-2^{-j-1}} + o_k(1).
\end{align*}
The product is convergent and it equals approximately $0.487$. It completes the proof.
\end{proof}

\section{Discussion}

An immediate consequence of Theorems \ref{thm-upp} and \ref{thm-low} is the following fact.

\begin{corrolary}
The constant
$$\hat{c} = \lim_{k\to\infty}\lim_{N\to\infty}\sup_{q_1>N}(A_\rho/M_\rho)^{2^{-k}}$$
exists and it is positive.
\end{corrolary}

The computing of $\hat{c}$ is definitely challenging. At this moment we only know that $0.487 < \hat{c} < 0.9541$, where the second inequality is proved in \cite{Bzdega-Height}. Even more challenging would by proving the analogous result for cyclotomic polynomials.

\section*{Acknowledgments}

The author would like to thank Wojciech Gajda for his remarks on this note.

\end{document}